\documentclass[11pt,english]{amsart}

\usepackage{amsmath,leftidx,amsthm}
\usepackage[all]{xy}
\usepackage{xspace,amsthm}
\usepackage[psamsfonts]{amssymb}
\usepackage[latin1]{inputenc}
\usepackage{graphicx,color}
\usepackage{hyperref,fancyhdr,esint}

\newtheorem{theorem}{\bf Theorem}
\newtheorem{proposition}[theorem]{\bf Proposition}

\newtheorem*{Theorem}{\bf Main Theorem}

\newtheorem{lemma}[theorem]{\bf Lemma}

\def\C{{\mathbb C}}
\def\N{{\mathbb N}}

\def\P{\mathbb{P}}

\def\bif{\textup{bif}}

\def\and{{\quad\text{and}\quad}}

\title{Parametric Lyapunov exponents}
\author{Henry De Th\'elin}
\address{Université Sorbonne Paris Nord, LAGA, CNRS, UMR 7539, F-93430, Villetaneuse, France}
\email{dethelin@math.univ-paris13.fr}
\author{Thomas Gauthier}
\address{CMLS UMR 7640, Ecole Polytechnique, Institut Polytechnique de Paris, 91128 Palaiseau Cedex, France}
\email{thomas.gauthier@polytechnique.edu}
\author{Gabriel Vigny}
\address{LAMFA UMR 7352, Universit\'e de Picardie Jules Verne, 33 rue Saint-Leu, 80039 Amiens Cedex 1, France}
\email{gabriel.vigny@u-picardie.fr}
\thanks{{Mathematics~Subject~Classification~(2010):}
28D20, 37F45, 37F10.}
\thanks{The second and third authors are partially supported by the ANR grant Fatou ANR-17-CE40-0002-01.}

\begin{document}

\begin{abstract} In an algebraic family of rational maps of $\P^1$, we show that, for almost every parameter for the trace of the bifurcation current of a marked critical value, the critical value is Collet-Eckmann. This extends previous results of  Graczyk and \'{S}wi\c{a}tek in the unicritical family, using Makarov theorem. Our methods are based instead on ideas of laminar currents theory.
\end{abstract}

\maketitle

\section{Introduction}

Let $\Lambda$ be a smooth complex quasi-projective variety and $f:\Lambda\times \P^1 \to \Lambda\times \P^1$ an \emph{algebraic family} of rational maps of degree $d\geq 2$: $f$ is a morphism and for each $(\lambda,z)$, $f(\lambda,z)=(\lambda,f_\lambda(z))$ where $f_\lambda$ is a rational map of $\P^1$ of degree $d$. Let also $a$  be a \emph{marked point}, i.e. a rational function $a:\Lambda\to\P^1$. A particularly interesting case is when $a$ is a marked critical point. A fundamental notion in complex dynamics is the notion of \emph{stability}: the point $a$ is stable at some parameter $\lambda_0$ if the sequence $\lambda \mapsto (f^n_\lambda(a(\lambda)))_n$ is normal in some neighborhood of $\lambda_0$. The \emph{bifurcation locus} of $a$ is then the set of unstable parameters. 

One can give a measurable sense to bifurcation using the \emph{bifurcation (or activity) current} of the pair $(f,a)$. It is the closed positive $(1,1)$-current $T_{f,a}:=(\Pi_\Lambda)_*(\widehat{T}\wedge[\Gamma_a])$,
where $\widehat{T}$ is the fibered Green current of the family $f$, $\Gamma_a$ is the graph of $a$ and $\Pi_\Lambda:\Lambda\times\P^1\to\Lambda$ is the canonical projection. This current is supported by the bifurcation locus of the marked point $a$,  see e.g. \cite{favredujardin}. When $\dim(\Lambda)=1$, then $T_{f,a}$ is a measure that we simply denote $\mu_{f,a}$. 

\bigskip
In some sense, the bifurcation current is a parametric analogous of the Green current of an endomorphism of $\P^k$ which measures the dynamical unstability. As such, it is interesting to develop an ergodic theory for the bifurcation currents. This is what we did in \cite{bif-entropy} where we defined a notion of parametric entropy and proved, e.g.,  that in a one dimensional family, the measure $\mu_{f,a}$ is a measure of maximal entropy. Pursuing our study, in the present article, we address the notion of parametric Lyapunov exponent. 

An historically important example is the Mandelbrot set in the unicritical family: $f_\lambda(z)=z^d+\lambda$ with $\lambda\in\C$ and $a(\lambda):=\lambda$. In this case, the bifurcation measure $\mu_{f,a}$ is the equilibrium measure (or equivalently the harmonic measure) $\mu_{\mathrm{M}_d}$ of the  degree $d$ Mandelbrot set $\mathrm{M}_d$. In this context,
Graczyk and \'Swi\c{a}tek \cite{GS-Mand} described the dynamics of a typical parameter:

\begin{theorem}[Graczyk-\'Swi\c{a}tek]
In the unicritical family of degree $d$, for $\mu_{\mathrm{M}_d}$-almost every parameter $\lambda\in\C$, we have
\[\lim_{n\to\infty}\frac{1}{n}\log|(f_\lambda^n)'(\lambda)|=\log d.\]
\end{theorem}

As the measure $\mu_{\mathrm{M}_d}$ has Hausdorff dimension $1$, this result may be reinterpreted as a parametric Ruelle (in)equality ``the Lyapunov exponent of $\mu_{\mathrm{M}_d}$ is equal to $\log d=h_{\bif}(f,a)/\dim\mu_{f,a}$''. 
Here, we generalize partially this result to the case of any pair $(f,a)$. For $\lambda\in \Lambda$, denote $\mathrm{Crit}(f_\lambda):=\{z\in \P^1, \ f_\lambda'(z)=0\}$.
Denote $\omega_{\Lambda}$ a Kähler form on $\Lambda$ so that $T_{f,a}\wedge \omega_{\Lambda}^{\dim(\Lambda)-1}$ is the \emph{trace measure}  of $T_{f,a}$. When $\dim(\Lambda)=1$, for a measure $\mu$ on an open set $U \subset \Lambda$, we define $D_U^*$, the \emph{upper packing dimension} of $\mu$ in $U$ as
\[D_U^*:= \mathrm{supess}_{\lambda \in U} \, \phi^*(\lambda)\leq 2,\] 
where, for $\lambda \in U$, 
\[\phi^*(\lambda):=\limsup_{r\to 0} \frac{\log \mu(B(\lambda,r))}{\log r}.  \]
We have the inequality $D^*\leq 2$ since the upper packing dimension of a measure is less than the dimension of the ambient space. We prove the following, where $f^\#$ is the spherical derivative.

\begin{Theorem}\label{tm:aeCE}
Let $f:\Lambda\times\P^1\to \Lambda\times\P^1$ be an algebraic family of rational maps of degree $d\geq2$ parametrized by a quasi-projective variety $\Lambda$ and let $a:\Lambda\to\P^1$ be a rational function for which there exists $\lambda_0 \in \Lambda$ such that $\{ f^n_{\lambda_0}(a(\lambda_0)), \ n\in\N\} \cap \mathrm{Crit}(f_{\lambda_0})=\varnothing$.  Then
\begin{itemize}
\item when $\dim(\Lambda)=1$, for any subset $U\subset \Lambda$, we have  
	\[\mu_{f,a}-\mathrm{a.e} \ \lambda \in U, \ \liminf_{n\to\infty}\frac{1}{n}\log(f_\lambda^n)^\#(a(\lambda))\geq \frac{\log d}{D_U^*}\geq \frac{1}{2}\log d \]
\item when $\dim(\Lambda)\geq1$, for almost every parameter $\lambda\in\Lambda$ with respect to the trace measure of $T_{f,a}$, we have \[\liminf_{n\to\infty}\frac{1}{n}\log(f_\lambda^n)^\#(a(\lambda))\geq \frac{1}{2}\log d.\]
\end{itemize}
\end{Theorem}
A particularly interesting case is when $a=\mathfrak{c}$ is a marked critical value which is not stably precritical. Then, the Theorem means that, for almost every parameter $\lambda$ with respect to the trace measure of $T_{f,\mathfrak{c}}$, $\mathfrak{c}(\lambda)$ is Collet-Eckmann. Hence, the Large Scale Condition of \cite{AGMV} is generic for the trace measure of $T_{f,\mathfrak{c}}$.

\medskip

Let us say a few words about the strategy of the proof. First, the proof of Graczyk and \'{S}wi\c{a}tek relies deeply on the fact that $\mu_{\mathrm{M}_d}$ is the harmonic measure of a fully connected compact set of the complex plane and on profound results of Makarov on the harmonic measures of such compact sets \cite{Makarov}. As such, it can not be used for arbitrary families of rational maps.

Instead, when $\dim\Lambda=1$, we construct here many disks in the graph of $f^n(a)$ (which is an analytic set of dimension $1$ in a $2$-dimensional space) using classical ideas of the theory of laminar currents (\cite{Dujardin_intersection, thelin_laminaire}). We then use those disks to bound the parametric Lyapunov exponent (see Theorem~\ref{parametric_lyapunov}). We then use a transversality argument to bound the dynamical Lyapunov exponent. Finally, we use Fubini Theorem to deduce the case where $\dim\Lambda\geq 1$.

Nevertheless, for the unicritical family, we do not get the optimal bound of Graczyk and \'{S}wi\c{a}tek since we do not know whether $D^*=1$ for the harmonic measure of the degree $d$ Mandelbrot set (Makarov tells us that $D_*=1\leq D^*$). Still, we show that the bound in Theorem~\ref{tm:aeCE} is sharp, in general, by considering a constant family of Lattès maps with a moving marked point.

\section{On a set of full measure}
In this section, $\Lambda$ is a smooth quasi-projective curve and $f:\Lambda\times \P^1 \to \Lambda\times \P^1$ an algebraic family of rational maps of degree $d\geq 2$.  Let $\omega_{\P^1}$ denote the Fubini-Study form on $\P^1$ and $\omega_\Lambda$ a volume form on $\Lambda$. Denote by $\Pi_{\P^1}:\Lambda\times \P^1 \to \P^1$ and $\Pi_{\Lambda}:\Lambda\times \P^1 \to \Lambda$ the canonical projections. Let $\widehat{\omega}_1:=\Pi_{\P^1}^*(\omega_{\P^1})$ and $\widehat{\omega}_2:=\Pi_{\Lambda}^*(\omega_{\Lambda})$. 
Let $\mu_{f,a}$ be the bifurcation (or activity) measure of $(f,a)$:
\[\mu_{f,a}:=\left(\Pi_{\Lambda}\right)_*\left([\Gamma_a]\wedge \widehat{T}\right).\]
Recall that $\widehat{T}=\lim d^{-n} (f^n)^*(\widehat{\omega}_1)=\widehat{\omega}_1 +dd^c g$ where $g$ is a $\alpha$-Hölder $\widehat{\omega}_1$-psh function.  Let $U\subset \Lambda$. We are interested in the $\mu_{f,a}$-a.e value of the \emph{parametric Lyapunov exponent} defined by
\[\liminf_{n\to\infty} \frac{1}{n} \log \left\|\frac{\partial f^n}{\partial\lambda}(\lambda,a(\lambda)) \right\|. \] 
Here, the norm is computed with respect to the spherical distance on $\P^1$, but, as any equivalent metric will give the same result, the exponent can be computed in some finite charts.  The purpose of this section is to prove 
\begin{theorem}\label{parametric_lyapunov} The parametric Lyapunov exponent satisfies
\[\mu_{f,a}-a.e. \ \lambda \in U, \quad  \liminf_{n\to\infty} \frac{1}{n} \log \left\|\frac{\partial f^n}{\partial\lambda}(\lambda,a(\lambda)) \right\|\geq \frac{\log d}{D_U^*}\geq\frac{\log d}{2} .\] 	
\end{theorem}
The rest of the section is devoted to the proof of the theorem. Observe that it is enough to restrict to the case where $U$ is a disk relatively compact in $\Lambda$ such that $\mu_{f,a}(U)>0$.  To simplify the notations, we write $D^*$ instead of $D_U^*$.

\subsection{Constructing disks in $f^n(\Gamma_a)$ }
Let $\varepsilon>0$. We fix $0< \beta \ll \log d$ and we construct disks of size $e^{-\beta n}$ in $f^n(\Gamma_a)$. We will use classical idea of the theory of laminar currents. 

 We let $\mathcal{C}$ be a finite cover of $\P^1$ given by charts $C$ where $C$ is the unit square in $\C$ centered at $0$. We also let $V$ be an open neighborhood of $\overline{U}$ which can be taken to be a square in $\C$ of size $1$. 

 We can assume that $\mu_{f,a}(V)\stackrel{\forall n}{=}\frac{\left[f^n(\Gamma_a)\right]}{d^n}\wedge \widehat{T} (V\times \P^1) \leq 1$ (up to restricting $U$ and $V$). We subdivide $C$ and $V$ into squares of size $e^{-\beta n}$ which gives us a subdivision of $C\times V$ into (4 dimensional) cubes of size $e^{-\beta n}$, we denote by $\mathcal{P}$ this tiling.  Let $0<\eta<1$. For $P\in \mathcal{P}$ of center $c(P)$, let $P^\eta$ be the image of $P$ by the homothety of ratio $\eta$ and center $c(P)$.  Let $\mathcal{P}^\eta$ denote the union of the $P^\eta$. For $z \in P$, let $\mathcal{P}_z:= z-c(P)+\mathcal{P}$ denote the translation of $\mathcal{P}$ by the vector $z-c(P)$. Finally, let $\mathcal{P}_z^\eta$ denote the union of all the homothetics of elements of $\mathcal{P}_z$. Recall the following result (\cite[Lemme 4.5]{Dujardin_intersection}). 
\begin{lemma}\label{4.5} With the above notations, there exists $z\in P$ such that 
	\[ \frac{\left[f^n(\Gamma_a)\right]}{d^n}\wedge \widehat{T} (\mathcal{P}_z\backslash \mathcal{P}_z^\eta) \leq 2(1-\eta^4).    \]
\end{lemma}  
We take  \[\eta=(1- \exp(\frac{-\beta n \alpha}{4}))^{\frac{1}{4}} \]
so that $2(1-\eta^4)\leq 2  \exp(\frac{-\beta n \alpha}{4})$ (recall that $\alpha$ is the Hölder exponent of a quasi-potential of $\widehat{T}$). So we translate $C$ and $V$ by the $z$ given by the above lemma. Since $\mathcal{P}$ does not move much as $\mathrm{diam}(P)\leq \exp(-\beta n)$, this gives us new $C$ and $V$ that we still denote $C$ and $V$ since the collection of the new $C$ still covers $\P^1$ and $U\Subset V$ still holds. 

We now construct disks in $f^n(\Gamma_a) \cap \Lambda\times C$. Let $\chi$ denote the Euler characteristic. Then, $\chi(F^n(\Gamma_a))\geq \chi(\Gamma_a)=:\chi_0$ as the Euler characteristic increases by direct image. Let $R_n$ denote the number of ramifications of $(\Pi_{\P^1})|_{f^n(\Gamma_a)}$ and let $d_n$ the topological degree of  $(\Pi_{\P^1})|_{f^n(\Gamma_a)}$. Then $d_n=d^n \times d'$ where $d'$ is the topological degree of $a$. By Riemann-Hurwitz, we have $\chi(f^n(\Gamma_a))= d_n \chi(\P^1)-R_n $ so $R_n \leq c d^n$ for some constant $c$ that does not depend on $n$ nor $\beta$.
  
  Consider the set of connected components of all the preimages $(\Pi_{\P^1})|^{-1}_{f^n(\Gamma_a)}(S)$ where $S$ belongs to the above tiling of $C$ into squares of size $e^{-\beta n}$. We call \emph{island}  such a connected component $I$ for which $(\Pi_{\P^1})|_{f^n(\Gamma_a)}$ is a biholomorphism from $I$ to $S$. In particular, the sum of the degrees of the projection $(\Pi_{\P^1})_{|f^n(\Gamma_a)}$ restricted to each connected component which is not an island is $\leq c d^n$. Let us also remove the islands whose area is $\geq s(\varepsilon)$ ($s(\varepsilon)$ will be made explicit later). As $f^n(\Gamma_a)$ has area $\leq \mathrm{area}(\Gamma_a)d^n$, we have removed at most $ \frac{c d^n}{s(\varepsilon)}$ (taking a larger $c$ if necessary). Let us denote by $I^2_n$ the union of all the other islands, which are those we call \emph{good disks}. Let $B_n:=\left[ I_n^2 \right]/d^n$, then
  \begin{lemma}\label{size_of_good_disks}
  With the above notations, there exist $n_1\in \N$ and a constant $K'(\varepsilon)$ such that  
  	\[\forall n\geq n_1 ,\ \int_{V\times C}  \left(\frac{f^n(\Gamma_a)}{d^n}-B_n \right) \wedge \widehat{T} \leq K(\varepsilon)e^{ \frac{-\beta\alpha n }{4}}. \]
  \end{lemma}
\begin{proof}
First, observe that there exist $n_0 \in \N$ such that $\forall n\geq n_0$
\begin{align}\label{bound_in_mass}
\left\langle \frac{f^n(\Gamma_a)}{d^n} -B_n, \widehat{\omega}_1|_{C} +\widehat{\omega}_2|_{V}  \right\rangle &\leq     \left\langle \frac{f^n(\Gamma_a)}{d^n} -B_n, \widehat{\omega}_1|_{C}  \right\rangle +\left\langle \frac{f^n(\Gamma_a)}{d^n},  \widehat{\omega}_2|_{V} \right \rangle  \nonumber\\
& \leq \frac{cd^ne^{-2\beta n}}{d^n}+  \frac{ c e^{-2\beta n} d^n}{s(\varepsilon)d^n}  +\frac{1}{d^n}\nonumber \\ 
&\leq \frac{ 3c e^{-2\beta n}}{s(\varepsilon)}  \end{align}  
where we used that  $\beta \ll \log d$ and that $f^n(\Gamma_a)$ is a graph (hence the area of the projection on the first coordinate is the area of $V$). 
We now follow ideas of Dujardin (\cite{Dujardin_intersection}). Take a smooth cut-off function $\Psi$ which is equal to $1$ on $P^\eta$ and $0$ near $\partial P$ for every $P\in \mathcal{P}$ and such that there exists a constant $K$ independent of $n$ satisfying
\begin{align}\label{page_5}
 \|\Psi\|_{\mathcal{C}^2} &\leq K \left(\frac{1}{ (1-\eta)e^{-\beta n}}\right)^2 \nonumber\\
 	                       &\leq \frac{K e^{2\beta n}}{\left(1-\left(1-e^{\frac{-\beta n \alpha}{4}}\right)^\frac{1}{4}\right)^2} \simeq \frac{Ke^{2\beta n}}{\left(1-1+\frac{1}{4}e^{\frac{-\beta n \alpha}{4}}\right)^2} \nonumber \\
 	                      & \leq 20K e^{\beta n(2+ \frac{\alpha}{2})}
 \end{align}	
 for $n\geq n_0$. Writing as above $\widehat{T}=\widehat{\omega}_1+ dd^c g$ gives
 \begin{align*}
 \int_{V\times C}  \left(\frac{f^n(\Gamma_a)}{d^n}-B_n \right) \wedge \widehat{T} &\leq \int_{\mathcal{P}\backslash \mathcal{P}^\eta}  \left(\frac{f^n(\Gamma_a)}{d^n}-B_n \right) \wedge \widehat{T}+\int \Psi  \left(\frac{f^n(\Gamma_a)}{d^n}-B_n \right) \wedge \widehat{T} \\
                  &\leq 2 e^{-\frac{\beta n \alpha}{4}} +\int \Psi  \left(\frac{f^n(\Gamma_a)}{d^n}-B_n \right) \wedge \widehat{\omega}_1 \\
                   & \quad +\int \Psi  \left(\frac{f^n(\Gamma_a)}{d^n}-B_n \right) \wedge dd^c g \\
                   &\leq 2 e^{-\frac{\beta n \alpha}{4}}+  \frac{ 3c e^{-2\beta n}}{s(\varepsilon)} +\int  \Psi  \left(\frac{f^n(\Gamma_a)}{d^n}-B_n \right) \wedge dd^c g
  \end{align*} 
  where we used Lemma~\ref{4.5} and the bound \eqref{bound_in_mass}. For the last term, by Stokes ($c(P)$ denotes the center of the cube $P$):
  \begin{align*}
\int  \Psi  \left(\frac{f^n(\Gamma_a)}{d^n}-B_n \right) \wedge dd^c g &= \sum_{P\in \mathcal{P}} \int_P  \Psi  \left(\frac{f^n(\Gamma_a)}{d^n}-B_n \right) \wedge dd^c g  \\
                               &=\sum_{P\in \mathcal{P}} \int_P (g-g(c(P)))    \left(\frac{f^n(\Gamma_a)}{d^n}-B_n \right) \wedge  dd^c \Psi  \\
                               &\leq \sum_{P\in \mathcal{P}} \int_P |g-g(c(P))|    \left(\frac{f^n(\Gamma_a)}{d^n}-B_n \right) \\
                               & \quad \quad  \wedge  20K e^{\beta n(2+ \frac{\alpha}{2})} \left(\widehat{\omega}_1|_{C} +\widehat{\omega}_2|_{V}\right).
  \end{align*}
  Now, $ |g-g(c(P))| \leq c e^{-\beta \alpha n}$ since $g$ is $\alpha$-Hölder (we can take the same $c$ than in \eqref{bound_in_mass} up to increasing it) so that
 \begin{align*}
 \int  \Psi  \left(\frac{f^n(\Gamma_a)}{d^n}-B_n \right) \wedge dd^c g &\leq  c e^{-\beta \alpha n}  20K e^{\beta n(2+ \frac{\alpha}{2})} \int \left(\frac{f^n(\Gamma_a)}{d^n}-B_n \right) \\
  &\quad \quad  \wedge \left(\widehat{\omega}_1|_{C}+\widehat{\omega}_2|_{V}\right) 
 \\
   &\leq 20 Ke^{2\beta n} e^{-\beta \frac{\alpha}{2} n}  \frac{ 3c^2 e^{-2\beta n}}{s(\varepsilon)}\\
   &\leq K'(\varepsilon) e^{-\beta \frac{\alpha}{2} n}  
 \end{align*}  
 where we used \eqref{bound_in_mass}, $K'(\varepsilon)$ is a large enough constant and $n\geq n_0$. Combining all the above gives a rank $n_1 \geq n_0$ and a constant $K(\varepsilon)$ such that for $n\geq n_1$:
  \[\int_{V\times C}  \left(\frac{f^n(\Gamma_a)}{d^n}-B_n \right) \wedge \widehat{T} \leq 2 e^{-\frac{\beta n \alpha}{4}}+  \frac{ 3c e^{-2\beta n}}{s(\varepsilon)}+ K'(\varepsilon) e^{-\beta \frac{\alpha}{2} n} \leq K(\varepsilon)e^{ \frac{-\beta\alpha n }{4}}. \]
\end{proof}
Taking a finite cover of $\P^1$, we have the above estimate on $V \times \P^1$.

\subsection{Using the above disks to bound the Lyapunov exponent}
We first show that we can find an arbitrary large set in $U$ of parameters $\lambda$ for which the point-wise dimension of $\mu_{f,a}$ is controlled by $D^*$ and for which all the corresponding points $(\lambda, f^n_\lambda(a(\lambda)))$ belong to a good disk constructed above. Then, using Koebe's Distortion Theorem, we bound from below the parametric Lyapunov exponents. 
\begin{lemma}\label{size_good_parameters}
	With the above notations, there exist a set $W\subset U$, integers $n_2\in \N$ and $\ell_0\in \N$ such that  
	\begin{itemize}
		\item $\mu_{f,a}(U\backslash W) \leq \varepsilon$, 
		\item $\forall \lambda \in W, \ \forall n\geq n_2$,  $f^n(\lambda,z)$ belongs to a good disk $D\subset f^n(\Gamma_a)$ and  $f^n(\lambda,z)\notin \mathcal{P}\backslash \mathcal{P}^\eta$,
		\item $\forall \lambda \in W, \ \forall r<\frac{1}{\ell_0}$,   $ \mu_{f,a}(B(\lambda,r))\geq r^{D^*+\beta}$.
	\end{itemize}
\end{lemma}
\begin{proof}
	Take $n \geq n_2\geq n_1$ such that  
\[\sum_{C\in \mathcal{C} }\sum_{n\geq n_2} (2+K(\varepsilon))e^{ \frac{-\beta\alpha n }{4}} \leq \frac{\varepsilon}{2},\]
where  $\mathcal{C}$ denotes the finite cover of $\P^1$ defined in the previous section and $K(\varepsilon)$ is the constant given by Lemma~\ref{size_of_good_disks} (we can take the same constant $K(\varepsilon)$ for every $C$). Denote
\begin{align*} \widehat{A}:= &\{(\lambda,z)\in \left(U\times \P^1\right)\cap \Gamma_a, \ \forall n\geq n_2, f^n(\lambda,z)\ \mathrm{belongs\ to\ a\ good\ disk}\ D\subset f^n(\Gamma_a) \\
  & \quad \mathrm{and \ } f^n(\lambda,z)\notin \mathcal{P}\backslash \mathcal{P}^\eta  \} \quad \mathrm{and}\\
 \widehat{A}_n:= &\{(\lambda,z)\in \left(U\times \P^1\right)\cap \Gamma_a, \ f^n(\lambda,z)\ \mathrm{belongs\ to\ a\ good\ disk}\ D\subset f^n(\Gamma_a) \\ & \quad \mathrm{and \ } f^n(\lambda,z)\notin \mathcal{P}\backslash \mathcal{P}^\eta  \} 
\end{align*}
so that $\widehat{A}=\cap_{n\geq n_2} \widehat{A}_n$.  Then, using $f^*\widehat{T}=d\widehat{T}$ and Lemma~\ref{4.5}: 
\begin{align*} \left[\Gamma_a\right] \wedge \widehat{T}\left(\widehat{A}^c\cap \Pi_{\Lambda}^{-1} U\right) &= \left[\Gamma_a\right] \wedge \widehat{T}\left(  \left(\cup_{n\geq n_2} \widehat{A}_n^c \right)\cap \Pi_{\Lambda}^{-1} U\right) \\
 &\leq \ \sum_{n\geq n_2}  \frac{\left[f^n(\Gamma_a)\right]}{d^n} \wedge \widehat{T}\left(f^n(\widehat{A}_n^c)\cap \Pi_{\Lambda}^{-1} U\right)     \\
   &\leq \sum_{C\in \mathcal{C}}  \sum_{n\geq n_2}   \frac{\left[f^n(\Gamma_a)\right]}{d^n} \wedge \widehat{T}\left(f^n(\widehat{A}_n^c)\cap \left(U\times C\right)\right) \\
   &\leq  \sum_{C\in \mathcal{C}} \sum_{n\geq n_2}  \frac{\left[f^n(\Gamma_a)\right]}{d^n} \wedge \widehat{T}\left(\mathcal{P}\backslash\mathcal{P}^\eta \right)+ K(\varepsilon) e^{ \frac{-\beta\alpha n }{4}} \\
   &\leq  \sum_{C\in \mathcal{C}} \sum_{n\geq n_2} 2e^{-\frac{\beta \alpha n}{4}}+ K(\varepsilon) e^{ \frac{-\beta\alpha n }{4}} \leq \frac{\varepsilon}{2}.
\end{align*}
Hence $\widehat{T}\wedge[\Gamma_a]\left(\widehat{A}\right)\geq \mu_{f,a}(U)-\varepsilon/2$.
Now, recall that by definition of $D^*$, we can find $B\subset U$ such that $\mu_{f,a}(U\backslash B )=0$ and 
\[\forall \lambda\in B, \ \limsup_{r\to 0} \frac{\log \mu_{f,a}(B(\lambda,r))}{\log r} \leq D^*.\]
In particular, 
\[\forall \lambda \in B, \ \exists r_0, \ \forall r\leq r_0, \ \mu_{f,a}(B(\lambda,r))\geq r^{D^*+\beta}.\]
Let 
\[B_\ell:= \left\{ \lambda\in B, \ \forall r<\frac{1}{\ell},  \  \mu_{f,a}(B(\lambda,r))\geq r^{D^*+\beta}\right\}.\]
In particular, $\cup_\ell B_\ell =B$ and the union is increasing so that we can choose $\ell_0$ large enough so that $\mu_{f,a} (B_{\ell_0}) \geq \mu_{f,a}(U) - \varepsilon/2$. Then, the set $W:=B_{\ell_0}\cap \Pi_{\Lambda}\left(\widehat{A}\right)$ satisfies
\[\mu_{f,a}(W)\geq \mu_{f,a}(U)-\varepsilon \]
since $\mu_{f,a}\left(\Pi_{\Lambda}\left(\widehat{A}\right)\right)= \widehat{T}\wedge[\Gamma_a]\left( \Pi_\Lambda^{-1}\Pi_{\Lambda}(\widehat{A})\right)\geq  \widehat{T}\wedge[\Gamma_a]\left((\widehat{A})\right)\geq \mu_{f,a}(U)-\varepsilon/2$. This proves the lemma.
\end{proof}

Let $W$ be given by the above lemma and pick $\lambda \in W$. Let $n\geq n_2$, by definition, there exists a disk $D$ above a square $S$ of size $e^{-\beta_n}$ in the chart $C$ such that $(\lambda,a(\lambda))\in D$. As $f^n(\lambda,a(\lambda)) \notin \mathcal{P}\backslash\mathcal{P}^\eta$, then $\Pi_{\P^1}\left(f^n(\lambda,a(\lambda))\right) \in C^\eta$ (the homothetic of $C$ of ratio $\eta$ with respect to its center). Define
\[\eta':=\frac{1+\eta}{2}.\]  
Let $\Delta:= \Pi_{\P^1}^{-1}(C^{\eta'})\cap D$ and let $\Delta_n \subset \Gamma_a$ be the preimage of $\Delta$ by $f^n$ ($f^n$ is injective on $\Gamma_a$).  
\begin{lemma}\label{mass_of_disk} With the above notations, there exists an integer $n_3\geq n_2$ such that
	\[\forall n\geq n_3, \ \int \left(1_{B(\lambda, \frac{1}{\ell_0})}\circ \Pi_{\Lambda} \right)  \widehat{T}\wedge [\Delta_n] \leq  \frac{200K e^{\frac{\beta n\alpha}{2}} e^{2\beta n}}{d^n}  \]
\end{lemma}
\begin{proof}
Let $\psi$ be a smooth cut-off function on $\P^1$ which is equal to $1$ on $C^{\eta'}$ and $0$ near $\partial C$ and $\varphi$ be a smooth cut-off function on $\Lambda$ which is equal to $1$ on $B(\lambda, \frac{1}{\ell_0})$ and $0$ near $\partial B(\lambda, \frac{2}{\ell_0})$ so that ($K$ is a universal constant) \begin{itemize}
	\item  $d\psi\wedge d^c \psi \leq  K ((1-\eta')e^{-\beta n})^{-2} \omega_{\P^1}$ and $ dd^c \psi -K ((1-\eta')e^{-\beta n})^{-2} \omega_{\P^1} \leq 0$. 
\item $d\varphi\wedge d^c \varphi \leq  K (\ell_0)^{-2} \omega_{\Lambda}$ and 
$ dd^c \psi -K \ell_0^{-2} \omega_{\Lambda}\leq 0$.
\end{itemize}
 Then, by Stokes
\begin{align*}
\int_{\Delta \cap \Pi_\Lambda^{-1}(B(\lambda, \frac{1}{l_0}))}\widehat{T}  &\leq \int   \Pi^*_{\P^1}(\psi) \Pi_{\Lambda}^*(\varphi)[D]\wedge \widehat{T} \\
&\leq  \int \Pi^*_{\P^1} (\psi) \Pi_{\Lambda}^*(\varphi) [D]\wedge \widehat{\omega}_1 +  \int \Pi^*_{\P^1}(\psi) \Pi_{\Lambda}^*(\varphi) \psi [D]\wedge dd^c g  \\
& \leq \int \Pi^*_{\P^1}( \psi) [D]\wedge \widehat{\omega}_1 + \int  g [D]\wedge dd^c\left( \Pi^*_{\P^1} (\psi ) \Pi_{\Lambda}^*(\varphi)\right)\\
&\leq 1 +  \int  g [D]\wedge dd^c\left( \Pi^*_{\P^1} (\psi ) \Pi_{\Lambda}^*(\varphi)\right).
\end{align*}	
Now,  
\begin{align*} dd^c\left( \Pi^*_{\P^1} (\psi ) \Pi_{\Lambda}^*(\varphi)\right)&= \Pi_{\Lambda}^*(\varphi) dd^c\left( \Pi^*_{\P^1} (\psi ) \right)+\Pi^*_{\P^1} (\psi ) dd^c\left(  \Pi_{\Lambda}^*(\varphi)\right)+ \Pi_{\Lambda}^*(\varphi)  \Pi^*_{\P^1} (d\psi \wedge d^c \psi ) \\
 &\quad + \Pi_{\P^1}^*(\psi) \Pi_{\Lambda}^*(d\varphi\wedge d^c \varphi). 
\end{align*}
So we have the bound
\begin{align*}
\int_{\Delta \cap \Pi_\Lambda^{-1}(B(\lambda, \frac{1}{l_0}))} \widehat{T}  &\leq 1 +2\|g\|_{\infty}  \int_{\Pi_\Lambda^{-1}(B(\lambda, \frac{2}{l_0}))}  [D]\wedge K ((1-\eta')e^{-\beta n})^{-2} \widehat{\omega}_1\\
&\quad +2\|g\|_{\infty}  \int_{\Pi_\Lambda^{-1}(B(\lambda, \frac{2}{l_0}))}  [D]\wedge K \ell_0^{-2} \widehat{\omega}_2 \\
 &\leq 1+ 2\|g\|_\infty K (4((1-\eta')e^{-\beta n})^{-2}+ 4)\\
 &\leq 200K e^{\frac{\beta n\alpha}{2}} e^{2\beta n} 
\end{align*}
where we used the computations in \eqref{page_5} and assume $n\geq n_3\geq n_2$. In particular, using that $f^n(\Delta_n)=\Delta$ and the fact that $f^n$ is injective on $\Gamma_a$  gives:
\[ \int \left(1_{B(\lambda, \frac{1}{\ell_0})}\circ \Pi_{\Lambda} \right)  \widehat{T}\wedge [\Delta_n] =\frac{1}{d^n} \int   \left(1_{B(\lambda, \frac{1}{\ell_0})}\circ \Pi_{\Lambda} \right)  \widehat{T}\wedge [\Delta]\leq  \frac{200K e^{\frac{\beta n\alpha}{2}}e^{2\beta n}  }{d^n}   .     \] 
\end{proof}

\begin{lemma}\label{plusdidee}
For $\lambda \in W$, we have that
\[\liminf_{n\to \infty} \frac{1}{n} \log  \left\| \frac{\partial f^n}{\partial \lambda}(\lambda,a(\lambda))\right\| \geq \frac{\log d-\frac{\beta\alpha}{2} -2\beta}{D^*+\beta} -\frac{\beta \alpha}{4}-\beta. \]
\end{lemma}
\begin{proof}
	Let $r(\lambda)$ be the largest $r\geq 0$ such that $B(\lambda,r)\subset \Pi_{\Lambda}(\Delta_n)$. We now pick $s(\varepsilon):=\frac{\pi}{\ell_0^2}$ ($\ell_0$ only depends on $\varepsilon$). So, by definition, this means that $r(\lambda)\leq \frac{1}{\ell_0}$. 
	
	Since $\Pi_{\P^1}(f^n(\lambda,a(\lambda)))\in C^\eta$, there exists a disk $D_0$ of radius $(\eta'-\eta)e^{-\beta n}$ centered at  $\Pi_{\P^1}(f^n(\lambda,a(\lambda)))$ and contained in $C^{\eta'}$. The holomorphic map 
	\[h:\lambda \mapsto \Pi_{\P^1}(f^n(\lambda,a(\lambda)))\]
	is injective on $\Pi_{\Lambda}(\Delta_n)$ ($\lambda\mapsto (\lambda,a(\lambda))$ is injective, $f^n$ is injective on $\Gamma_a$ and $\Pi_{\P^1}$ is injective on $D$ since $D$ is a graph). Koebe $\frac{1}{4}$-Theorem implies that $h^{-1}(D_0)$ contains a disks of center $\lambda$ and radius 
	\[\frac{\left|(h^{-1})'\left( \Pi_{\P^1}\left(f^n\left(\lambda,a\left(\lambda\right)\right)\right) \right) \right|(\eta'-\eta)e^{-\beta n}}{4} \geq \frac{\left|(h^{-1})'\left( \Pi_{\P^1}\left(f^n\left(\lambda,a\left(\lambda\right)\right)\right) \right) \right|e^\frac{-\beta  \alpha n}{4} e^{-\beta n}}{32}.  \]
	By definition of $r(\lambda)$, we have
	\[r(\lambda)\geq \frac{\left|(h^{-1})'\left( \Pi_{\P^1}(f^n(\lambda,a(\lambda))) \right) \right|e^\frac{-\beta  \alpha n}{4} e^{-\beta n}}{32}= \frac{e^\frac{-\beta  \alpha n}{4} e^{-\beta n}}{32\left|h'(\lambda)\right|} .\]
	So, by Lemma~\ref{mass_of_disk} and the definition of $W$:
	\begin{align*}
	\frac{e^\frac{-\beta  \alpha n}{4} e^{-\beta n}}{32\left|h'(\lambda)\right|}&\leq r(\lambda) \leq \left(\mu_{f,a}(B(\lambda,r(\lambda))\right)^{\frac{1}{D^*+\beta}} \\
	&\leq  \left( \int \left(1_{B(\lambda, \frac{1}{\ell_0})}\circ \Pi_{\Lambda} \right)  \widehat{T}\wedge [\Delta_n] \right)^{\frac{1}{D^*+\beta}}  \leq \left( \frac{200K e^{\frac{\beta n\alpha}{2}}  e^{2\beta n} }{d^n} \right)^{\frac{1}{D^*+\beta}}.
	\end{align*}
	In other words
	\begin{equation}\label{tout_ca_pour_ca}
	\frac{e^\frac{-\beta  \alpha n}{4} e^{-\beta n}}{32} \left( \frac{d^n}{200K e^{\frac{\beta n\alpha}{2}} e^{2\beta n} } \right)^{\frac{1}{D^*+\beta}} \leq \left|h'(\lambda)\right|.
	\end{equation}
	By the chain rule 
	\[ |h'(\lambda)|= \left\|D\Pi_{\P^1}(f^n(\lambda,a(\lambda)))\circ \frac{\partial f^n}{\partial \lambda}(\lambda,a(\lambda))\right\| \leq \left\| \frac{\partial f^n}{\partial \lambda}(\lambda,a(\lambda))\right\|\]
	since projection are $1$-Lipschitz. Then, taking the logarithm in \eqref{tout_ca_pour_ca}, dividing by $n$ and letting $n\to \infty$ gives
	\[\liminf_{n\to \infty} \frac{1}{n} \log  \left\| \frac{\partial f^n}{\partial \lambda}(\lambda,a(\lambda))\right\| \geq \frac{\log d-\frac{\beta\alpha}{2} -2\beta }{D^*+\beta} -\frac{\beta \alpha}{4}-\beta,\]
	as required.
\end{proof}
\par\noindent Now, the proof of Theorem~\ref{parametric_lyapunov} is complete by taking $\beta \to 0$ and $\varepsilon \to 0$ in Lemmas~\ref{size_good_parameters} and \ref{plusdidee}.

\section{The proof of the Main Theorem}

\subsection{Comparing parameter and dynamical growth}
Here, we prove the following, relying on ideas of \cite{Astorg, AGMV}.
\begin{proposition}\label{prop:transfer}
Let $f:\Lambda\times\P^1\to\Lambda\times\P^1$ be an analytic family of degree $d$ rational maps. Assume that, for some parameter $\lambda_0$, there exists $\alpha>0$ such that
\[\liminf_{n\to\infty}\frac{1}{n}\log\left\|\frac{\partial f^n}{\partial\lambda}(\lambda_0,a(\lambda_0)) \right\| \geq \alpha>0.\]
Assume in addition that $f_{\lambda_0}^k(a(\lambda_0))\notin\mathrm{Crit}(f_{\lambda_0})$ for all $k\geq0$. Then we have
\[\liminf_{n\to\infty}\frac{1}{n}\log (f_{\lambda_0}^n)^\#(a(\lambda_0)) \geq\alpha.\]
\end{proposition}

For the sake of simplicity, we set $a_n(\lambda):=f_\lambda^n(a(\lambda))$ so that $a_0(\lambda)=a(\lambda)$. We also let $\dot{f}:=\partial_\lambda f(\lambda,\cdot)|_{\lambda=\lambda_0}$ and $\dot{a}:=\partial_\lambda a (\lambda_0)$.
The following is Lemma 4.4 in \cite{AGMV}
\begin{lemma}\label{lm:transfer-derivative}
Pick any parameter $\lambda_0$ and any integer $n\geq1$. As soon as we have that $f_{\lambda_0}'(f_{\lambda_0}^k(a(\lambda_0)))\neq0$ for all $0\leq k\leq n$, the following holds
\[\partial_\lambda a_n(\lambda_0)=(f_{\lambda_0}^n)'(a(\lambda_0))\cdot\left( \partial_\lambda a(\lambda_0)+\sum_{k=0}^{n-1} \frac{\dot{f}(a_k(\lambda_0)) }{ (f_{\lambda_0}^{k+1})'(a(\lambda_0))}\right).\]
\end{lemma}
\begin{proof}[Proof of Proposition~\ref{prop:transfer}] 
Since the coordinate on $\Lambda$ is a local coordinate and the metric on $\P^1$ is the one induced by the spherical distance:  
\begin{align*}
\|\partial_\lambda a_n(\lambda_0) \| &\stackrel{\forall n}{=}\frac{|\partial_\lambda a_n(\lambda_0)|}{1+|a_n(\lambda_0)|^2} \quad \
\|\dot{f}(a_k(\lambda_0)) \|\stackrel{\forall k}{=}\frac{|\dot{f}(a_k(\lambda_0))| }{1+|a_{k+1}(\lambda_0)|^2} \\
(f_{\lambda_0}^k)^\#(a(\lambda_0)) &=\frac{|(f_{\lambda_0}^k)'(a(\lambda_0))| (1+|a(\lambda_0)|^2) }{1+|a_k(\lambda_0)|^2}.
\end{align*}
Note that $\|\dot{f}(.)\|$ is continuous on $\mathbb{P}^1(\C)$ which is compact, so there exists $C_1\geq1$ such that $\|\dot{f}(z)\|\leq C_1$ for all $z\in\mathbb{P}^1(\C)$. Up to increasing $C_1$, we can assume $\|\partial_\lambda a(\lambda_0)\|\leq C_1$ as well. So, Lemma~\ref{lm:transfer-derivative} implies:
\begin{align*}
\frac{|\partial_\lambda a_n(\lambda_0)|}{1+|a_n(\lambda_0)|^2}&\leq 
\frac{|(f_{\lambda_0}^n)'(a(\lambda_0))| (1+|a(\lambda_0)|^2) }{1+|a_n(\lambda_0)|^2} \times \\
 &\ \left( \frac{|\partial_\lambda a(\lambda_0)|}{1+|a(\lambda_0)|^2}    + \sum_{k=0}^{n-1} \frac{|\dot{f}(a_k(\lambda_0))| (1+|a_{k+1}(\lambda_0)|^2)  }{ |(f_{\lambda_0}^{k+1})'(a(\lambda_0))|(1+|a_{k+1}(\lambda_0)|^2)(1+|a(\lambda_0)|^2)  }                   \right)
\end{align*}
Hence
\begin{align}\label{majoration_para_dyna}
\|\partial_\lambda a_n(\lambda_0) \|\leq 
(f_{\lambda_0}^n)^\#(a(\lambda_0)) \left( \|\partial_\lambda a(\lambda_0)\| + \sum_{k=0}^{n-1} \frac{\|\dot{f}(a_k(\lambda_0))\|}{ (f_{\lambda_0}^{k+1})^\#(a(\lambda_0)) }                   \right) \nonumber \\
\leq C_1(f_{\lambda_0}^n)^\#(a(\lambda_0)) \left(1+ \sum_{k=0}^{n-1} \frac{1}{ (f_{\lambda_0}^{k+1})^\#(a(\lambda_0)) }                   \right).
\end{align}
We first prove $\gamma:=\liminf_{n\to\infty}\frac{1}{n}\log (f^n_{\lambda_0})^\#(a(\lambda_0)) >-\infty$ by contradiction. If not, take $M\gg 1$ and let $n_0$ be the first integer such that $(f^{n_0}_{\lambda_0})^\#(a(\lambda_0)) \leq e^{-n_0 M}$. Taking $M$ larger will only increase $n_0$ so, by hypothesis, we can assume $\|\partial_\lambda a_{n_0}(\lambda_0)\|\geq \exp( n_0 \alpha/2)$.  Then, \eqref{majoration_para_dyna} gives:
\begin{align*}
 e^{n_0 \alpha/2}\leq \|\partial_\lambda a_{n_0}(\lambda_0)\| & \leq C_1\cdot  (f_{\lambda_0}^{n_0})^\#(a(\lambda_0))\cdot\left(1+\sum_{k=0}^{n_0-1}\frac{1}{(f_{\lambda_0}^{k+1})^\#(a(\lambda_0))}\right)\\
& \leq C_1 e^{ -n_0 M} \left(1+\sum_{k=1}^{n_0} e^{ k M}\right) \leq \frac{2C_1 e^M }{e^M-1}  
\end{align*}
which is impossible, so $\gamma>-\infty$.

We now prove similarly that $\gamma>0$. Assume by contradiction that $\gamma\leq0$ and fix $0<\varepsilon<\alpha/3$ and let $n_0\geq1$ be such that $\frac{1}{n}\log (f^n_{\lambda_0})^\#(a(\lambda_0))\geq \gamma-\varepsilon$ for all $n\geq n_0$. Set
\[C_2:=\max\left\{1,\max_{k\leq n_0}\frac{1}{ (f_{\lambda_0}^{k+1})^\#(a(\lambda_0))}\right\}<+\infty.\]
Taking $n_1\geq n_0$ large enough, we can assume that for all $n\geq n_1$,
\begin{enumerate}
\item $\frac{1}{n}\log\|\partial_\lambda a_n(\lambda_0)\|\geq \alpha-\varepsilon$, and
\item $\frac{1}{n}\log\left(3n C_1C_2/(\exp(-\gamma+\varepsilon)-1)\right)\leq\varepsilon$.
\end{enumerate}
We apply again \eqref{majoration_para_dyna}: for all $n\geq n_1$, we have
\begin{align*}
\|\partial_\lambda a_n(\lambda_0)\| & \leq C_1\cdot (f_{\lambda_0}^n)^\#(a(\lambda_0))\cdot\left(1+n_0C_2+\sum_{k=n_0+1}^{n-1}\exp((k+1)(-\gamma+\varepsilon))\right)\\
& \leq (f_{\lambda_0}^n)^\#(a(\lambda_0))\cdot \left(3nC_1C_2\right)\cdot\left(\frac{\exp((n+1)(-\gamma+\varepsilon))}{\exp(-\gamma+\varepsilon)-1}\right).
\end{align*}
By the choice of $n_1$, for all $n\geq n_1$ this gives
\begin{align*}
\alpha-\varepsilon & \leq \frac{1}{n}\log (f_{\lambda_0}^n)^\#(a(\lambda_0))+ \frac{1}{n}\log\left(\frac{3nC_1C_2}{\exp(-\gamma+\varepsilon)-1}\right)+\frac{n+1}{n}(-\gamma+\varepsilon)\\
& \leq \frac{1}{n}\log (f_{\lambda_0}^n)^\#(a(\lambda_0))+ \frac{n+1}{n}(-\gamma+2\varepsilon).
\end{align*}
Taking the $\liminf$ as $n\to\infty$ yields $\alpha-\varepsilon \leq \gamma-\gamma+2\varepsilon$, whence $\alpha\leq 3\varepsilon$. This is a contradiction. We thus have proved that $\gamma>0$.

\medskip

To conclude, we have to prove $\gamma\geq\alpha$. Using again \eqref{majoration_para_dyna}, we have
\[\varepsilon_n:=\frac{1}{n}\log\|\partial_\lambda a_n(\lambda_0)\|-\frac{1}{n}\log(f_{\lambda_0}^n)^\#(a(\lambda_0))\leq \frac{1}{n}\log C_1 \left| 1+\sum_{k=0}^{n-1} \frac{1}{(f_{\lambda_0}^{k+1})^\#(a(\lambda_0))}\right|.\]
Now, $\limsup_n \varepsilon_n\leq0$ since, as $\gamma>0$, the series $\sum_{k=0}^{+\infty} \frac{1}{(f_{\lambda_0}^k)^\#(a(\lambda_0))}$ is absolutely convergent.
\end{proof}

\subsection{Proof of the Main Theorem}

The case when $\dim\Lambda=1$ is just the combination of Theorem~\ref{parametric_lyapunov} and Proposition~\ref{prop:transfer}.

\medskip

We now assume $\dim\Lambda>1$. Let $\iota:\Lambda\hookrightarrow\P^N$ be an embedding of $\Lambda$ into a complex projective space, let $k:=\dim \Lambda<N$ and let $X$ be the intersection of the closure $\bar{\Lambda}$ be the closure of $\Lambda$ in $\P^N$ with the hyperplane at infinity $H_\infty:=\{Z_{N}=0\}$ in a given system of homogeneous coordinates $[Z_0:\cdots:Z_N]$ on $\P^N$. Let $Y$ be a linear subspace of $H_\infty$ of dimension $N-k$ so that $Y\cap X$ is a finite subspace and let $\mathcal{W}$ be the collection of all linear subspaces of $\P^N$ of dimension $N-k+1$ which intersect $H_\infty$ along $Y$. For any $W\in\mathcal{W}$, let
\[\Lambda_W:=\Lambda\cap W.\]
The variety $\Lambda_W$ is a quasi-projective curve. Let $f_W$ be the restriction of the family $f$ to a family parametrized by $\Lambda_W$ and let $\mu_W$ be the slice of $T_{f,a}$ along $\Lambda_W$, i.e. $\mu_W=T_{f,a}\wedge[\Lambda_W]$. According to Theorem~\ref{parametric_lyapunov}, for any $W$, and for $\mu_W$-almost every $\lambda\in W$, we have
\[\liminf_{n\to\infty} \frac{1}{n} \log \left\|\frac{\partial f^n}{\partial\lambda}(\lambda,a(\lambda)) \right\|\geq \frac{\log d}{2} .\] 
By hypothesis, the set of parameters $\lambda\in\Lambda$ such that there exists $k\geq0$ with $f_\lambda^k(a(\lambda))\in\mathrm{Crit}(f_\lambda)$ is a pluripolar subset of $\Lambda$. In particular, for Lebesgue almost every $W$, it intersects $W$ along a pluripolar set. As $\mu_W$ has continuous potentials, it does not give mass to pluripolar sets and Proposition \ref{prop:transfer} implies that
\[\liminf_{n\to\infty} \frac{1}{n} \log (f^n_\lambda)^\#(a(\lambda)) \geq\frac{\log d}{2},\] 
for $\mu_W$-almost every $\lambda \in W$, and for almost every $W$. The conclusion follows by Fubini Theorem.

\subsection{Sharpness of the bound}
Finally, we prove that the bound from below of the Main Theorem is sharp in the following simple situation. Take a constant family
\[\begin{cases}
f:\P^1 \times \P^1 &\to \P^1 \times \P^1 \\
        \quad \quad (\lambda,z)        &\mapsto (\lambda, f_0(z))
\end{cases} \]
where $f_0$ is a Lattès map of degree $d$ and take $a:\Lambda\to \Lambda$ be the marked point defined by $a(\lambda)=\lambda$. Then, one has $\mu_{\bif, a}=\mu_{f_0}$ where $\mu_{f_0}$ is the maximal entropy measure of $f_0$. It is well known that $\mu_{f_0}$ is absolutely continuous with respect to the Lebesgue measure (so $D_U^*=2$ for any non empty open set $U\subset \Lambda=\P^1$) and its Lyapunov exponent is $\log d/2$ (\cite{Zdunik}). This means, in particular, that Theorem\ref{tm:aeCE} is sharp here, since for  $\mu_{f_0}$-a.e. $\lambda$ in $U$
\[\limsup \frac{1}{n}\log (f_\lambda^n)^\#(a(\lambda))\leq  \frac{1}{D_U^*}\log d.\]

\end{document}